\documentclass[11pt,leqno]{article}
\usepackage{graphicx, amsfonts, amsthm, amsxtra, amssymb, verbatim,latexsym,bbm}
\usepackage{calrsfs}
\usepackage{addfont}
\usepackage{hyperref}
\usepackage[mathscr]{euscript}
\usepackage{color}   
\usepackage[all,cmtip]{xy}

\begin{document}

\def\Z{\mathbb{Z}}                   
\def\Q{\mathbb{Q}}                   
\def\C{\mathbb{C}}                   
\def\N{\mathbb{N}}                   
\def\R{\mathbb{R}}                   
\def\H{{\mathbb{H}}}                 

\def\Ra{{\sf R}}                      
\def\T{{\sf T}}                      
\def\t{{\sf t}}                       

\def\tr{{{\mathsf t}{\mathsf r}}}                 

\def\rc{{Rankin-Cohen }}
\def\rs{{Ramanujan-Serre type derivation }}
\def\DHR{{\rm DHR }}            

\def\P {\mathbb{P}}                  
\def\E{{\cal E}}
\def\L{{\sf L}}             
\def\CX{{\cal X}}
\def\dt{{\sf d}}             
\def\LG{{\sf G}}   
\def\LA{{\rm Lie}(\LG)}   
\def\amsy{\mathfrak{G}}  
\def\gG{{\sf g}}   
\def\gL{\mathfrak{g}}   
\def\rvf{{\sf H}}   
\def\cvf{{\sf F}}   
\def\di2{{ m}}   
\def\a{\mathfrak{a}}
\def\b{\mathfrak{b}}

\def\mmat{e}
\def\cmat{f}
\def\rmat{h}
\def\sl2{\mathfrak{sl}_2(\C)}
\def\SL2{{\rm SL}_2(\Z)}
\def\qmfs{\widetilde{\mathscr{M}}}             
\def\mfs{{\mathscr{M}}}                            
\def\cfs{{\mathscr{S}}}                            
\def\cyqmfs{{ \mathcal{M}}}             
\def\cymfs{{\mathcal{M}^2}}                            
\def\Ramvf{{\sf Ra}}             
\def\rcv{{\sf D}}             
\def\rcdo{{\mathcal{D}}}             
\def\rcdomf{{\mathscr{D}}}
\def\mdo{{\mathcal{R}}}             
\def\rdo{{\mathcal{H}}}             
\def\cdo{{\mathcal{F}}}             
\def\rsdo{{\partial}}             

\def\G0g{{\Gamma_0(N_g)}}             
\def\mdg{{\Delta_g}}             
\def\cmdg{{\delta_g}}             
\def\qmfg{{{E_g}_2}}             
\def\mfg{{\mathcal{E}_g}}             
\def\wmdg{{w_g}}             

\def\fvs{{t_1}}             
\def\svs{{t_2}}             
\def\tvs{{t_3}}             
\def\dfvs{{R^1}}             
\def\dsvs{{R^2}}             
\def\dtvs{{R^3}}             
\def\fss{{{\sf t}_1}}             
\def\sss{{{\sf t}_2}}             
\def\tss{{{\sf t}_3}}             
\def\wfss{{{\sf w}_1}}             
\def\wsss{{{\sf w}_2}}             
\def\wtss{{{\sf w}_3}}             
\def\wssj{{{\sf w}_j}}             
\def\wss{{{\sf w}}}             

\def\md{{\Delta}}             
\def\wmd{{{\sf w}}}             
\def\mf{{\rm j}}             
\def\bmf{{\sf t}}             
\def\vs{{t}}             
\def\prs{{\psi}}             
\def\nvs{{r}}             

\def\Yuk{{\sf Y}}                     
\def\X{{\sf X}}                      
\def\spec{{\rm Spec}}            

\def\too{{\mathscr{P}}_{3}}                      
\def\tot{\mathscr{Q}_{3}}                       
\def\toth{\mathscr{R}_{3}}                       
\def\tof{\mathscr{S}_{3}}                       

\def\tto{\mathscr{P}_{2}}                       
\def\ttt{\mathscr{Q}_{2}}                       
\def\ttth{\mathscr{R}_{2}}                       

\newtheorem{theo}{Theorem}[section]
\newtheorem{exam}{Example}[section]
\newtheorem{coro}{Corollary}[section]
\newtheorem{defi}{Definition}[section]
\newtheorem{prob}{Problem}[section]
\newtheorem{lemm}{Lemma}[section]
\newtheorem{prop}{Proposition}[section]
\newtheorem{rem}{Remark}[section]
\newtheorem{obs}{Observation}[section]
\newtheorem{conj}{Conjecture}
\newtheorem{nota}{Notation}[section]
\newtheorem{ass}{Assumption}[section]
\newtheorem{calc}{}
\numberwithin{equation}{section}

\begin{center}
{\LARGE\bf Ramanujan--type systems of nonlinear ODEs for $\Gamma_0(2)$ and $\Gamma_0(3)$ }
\\

\vspace{.25in} {\large {\sc
Younes Nikdelan \footnote{Max Planck Institute for Mthematics (MPIM), Vivatsgasse 7, 53111, Bonn, Germany. e-mail: nikdelan@mpim-bonn.mpg.de}
\footnote{Departamento de An\'alise Matem\'atica, Instituto de Matem\'atica
e Estat\'{i}stica (IME), Universidade do Estado do
Rio de Janeiro (UERJ), Rua S\~{a}o Francisco Xavier, 524, Rio de Janeiro, Brazil / CEP: 20550-900. e-mail: younes.nikdelan@ime.uerj.br}
}} \\
\end{center}
\vspace{.25in}
\begin{abstract}
This paper aims to introduce two systems of nonlinear ordinary differential equations whose solution components generate the graded algebra of quasi-modular forms
 on Hecke congruence subgroups $\Gamma_0(2)$ and $\Gamma_0(3)$.
 Using these systems, we provide the generated graded algebras with an $\sl2$-module structure.
 As applications, we introduce  Ramanujan-type tau functions for $\Gamma_0(2)$ and $\Gamma_0(3)$, and obtain some interesting and non-trivial recurrence and congruence relations.

\end{abstract}


\section{Introduction}

Besides the classical works of Darboux-Halphen \cite{da78,ha81} and Ramanujan \cite{ra16}, we can find  other works in the literature which are dedicated
to the study of systems of ordinary differential equations (ODEs)
whose solution components can be written in terms of quasi-modular forms, see for instance \cite{ohy96, zud03, mai11, dgms}. What makes the results we present here 
distinct from the other studies are the
origin of the systems which are considered here, and  the elaboration of the similarities in comparison with the classical case of Ramanujan system.
Indeed, in a geometric framework, the considered systems can be seen as vector fields, known as \emph{modular
vector fields}, in certain enhanced moduli spaces arising from the Dwork family.
In general, solution components of modular vector fields generate the space of \emph{Calabi-Yau modular forms}, which are interesting
objects to study, see \cite{GMCD-MQCY3,younes3}. In lower dimensions $1$ and $2$, which are studied in this paper, these spaces of Calabi-Yau modular forms
coincide with the spaces of classical quasi-modular forms on  $\Gamma_0(3)$ and $\Gamma_0(2)$, respectively. Hence, looking at the modular vector fields
from the viewpoint of this paper can reveal new ways of proceeding with the theory of Calabi-Yau modular forms.

Srinivasa Ramanujan in 1916, while working on arithmetic properties and the relationships between divisor, Gamma and Riemann zeta functions \cite{ra16},
encountered the system of nonlinear ODEs:
\begin{equation} \label{eq ramanujan}
 {\Ramvf}:\left \{ \begin{array}{l}
t_1'=\frac{1}{12}(t_1^2-t_2) \\
t_2'=\frac{1}{3}(t_1t_2-t_3) \\
t_3'=\frac{1}{2}(t_1t_3-t_2^2)
\end{array} \right..
\end{equation}
Here, and throughout the paper, "$\ast'$" refers to the derivation:
\[
 \ast':=q\frac{\partial \ast}{\partial q}=\frac{1}{2\pi i}\frac{d}{d \tau}, \ {\rm with } \ q=e^{2\pi i\tau}, \ \tau \in \C \ 
 {\rm and }\  {\rm Im} (\tau)>0 .
\]
He showed that the triple $(E_2,E_4,E_6)$ of the Eisenstein series forms a particular solution of the system $\Ramvf$, where for $j=1,2,3$:
\begin{align}
&E_{2j}(q)=1+b_j\sum_{k=1}^\infty \sigma_{2j-1}(k) q^{k}\ {\rm with} \ (b_1,b_2,b_3)=(-24,240,-504),\label{eq eis ser} \\
&\sigma_j(k)=\sum_{d\mid k}d^j\,.
\end{align}
The system of equations \eqref{eq ramanujan} are known as the \emph{Ramanujan relations between Eisenstein series}, and from now on we call them
the \emph{Ramanujan system}.
We recall the following known facts.
\begin{itemize}
  \item $E_4$ and $E_6$ are modular forms and $E_2$ is a quasi-modular form for $\SL2$.
  \item If we denote the space of the full modular forms and full quasi-modular forms, respectively, by $\mfs(\SL2)$ and $\qmfs(\SL2)$,
  then $\mfs(\SL2)=\C[E_4,E_6]$ and $\qmfs(\SL2)=\C[E_2,E_4,E_6]$.
  \item If we consider the vector field representation of the Ramanujan system, i.e. $\Ramvf=\frac{1}{12}(t_1^2-t_2)\frac{\partial}{\partial t_1}+\frac{1}{3}(t_1t_2-t_3)\frac{\partial}{\partial t_2}+\frac{1}{2}(t_1t_3-t_2^2)\frac{\partial}{\partial t_3}$, which is known as \emph{Ramanujan vector field} as well, then $\Ramvf$ along with the vector fields $\rvf:=2t_1\frac{\partial}{\partial t_1}+4t_2\frac{\partial}{\partial t_2}+6t_3\frac{\partial}{\partial t_3}$ and $\cvf:=-12\frac{\partial}{\partial t_1}$ forms a copy of the Lie algebra $\sl2$.
  \item The modular discriminant
\begin{equation}\label{eq modular dis}
\Delta:=\frac{1}{1728}(E_4^3-E_6^2)=\eta^{24}(q)\,,
\end{equation}
is a cusp form of weight $12$ for $\SL2$ that satisfies (one can check this directly or using \eqref{eq ramanujan}):
\begin{equation}\label{eq Delta/Delta'}
  \Delta'=E_2\Delta \ \ (\textrm{which is equivalent to } \eta'=\frac{1}{24} E_2\,\eta)\,,
\end{equation}
where $\eta$ is the classical eta function:
\begin{align}
&\eta({q})={q}^{\frac{1}{24}}\prod_{k=1}^\infty (1-{q}^{k})\,.
\end{align}
\end{itemize}

The main objective of this paper is to establish analogous results for the congruence subgroups $\Gamma_0(2)$ and $\Gamma_0(3)$ of 
$\SL2$. They are stated in the following theorems. In what follows, for any non-negative integer $k$, we denote by $\mfs_k(\Gamma)$ and
$\qmfs_k(\Gamma)$, respectively, the space of modular forms and the space of quasi-modular forms of weight $k$ for a subgroup
$\Gamma \subset \SL2$. Moreover, $\mfs(\Gamma):=\bigoplus_{k=0}^\infty \mfs_k(\Gamma)$ and $\qmfs(\Gamma):=\bigoplus_{k=0}^\infty \qmfs_k(\Gamma)$
stand for the graded algebra of modular forms and quasi-modular forms, respectively. 
\begin{theo} \label{thm main}
Consider the following system of ODEs:
\begin{equation}
 \label{eq mvf R2 3}
 \Ra_{2}:  \left \{ \begin{array}{l}
t_1'=\frac{1}{8}\big(t_1^2-t_2^2\big) \\
t_2'=\frac{1}{4}\big(t_1t_2-t_3\big)\\
t_3'=\frac{1}{2}\big(t_1t_3-t_2^3\big)
\end{array} \right. .
\end{equation}
\begin{enumerate}
  \item A particular solution of $\Ra_2$ is given as follows:
  \begin{equation} \label{eq solution R2 3}
\left \{ \begin{array}{l}
t_1=\tto(q):=\frac{1}{3}\big(E_2(q)+2E_2(q^2)\big),\\
t_2=\ttt(q):=2E_2(q^2)-E_2(q),\\
t_3=\ttth(q):=\frac{1}{3}\big(4E_4(q^2)-E_4(q)\big),
\end{array} \right.
\end{equation}
with $\ttt \in \mfs_2(\Gamma_0(2))$, $\ttth \in \mfs_4(\Gamma_0(2))$ and $\tto \in \qmfs_2(\Gamma_0(2))$. Furthermore, if we set
\[
\md_2:=\eta^8({q})\eta^8({q^2}),
\]
then $\md$ is a cusp form of weigh $8$ for $\Gamma_0(2)$, $\Delta_2=\frac{1}{256}(\ttt^4-\ttth^2)$ and
\[
\md_2'=\tto\md_2.
\]
  \item We have:
  \begin{align}
   & \mfs(\Gamma_0(2))=\C[\ttt,\ttth] \\
   & \qmfs(\Gamma_0(2))=\C[\tto,\ttt,\ttth]
  \end{align}
  \item The vector field $\Ra_2=\frac{1}{8}(t_1^2-t_2^2)\frac{\partial}{\partial t_1}+\frac{1}{4}(t_1t_2-t_3)\frac{\partial}{\partial t_2}
  +\frac{1}{2}(t_1t_3-t_2^3)\frac{\partial}{\partial t_3}$  along with the vector fields 
  $\rvf_2:=2t_1\frac{\partial}{\partial t_1}+2t_2\frac{\partial}{\partial t_2}+4t_3\frac{\partial}{\partial t_3}$ 
  and $\cvf_2:=-8\frac{\partial}{\partial t_1}$ is isomorphic to the Lie algebra $\sl2$.
       \end{enumerate}
\end{theo}

\begin{theo} \label{thm main2}
Consider the following system of ODEs:
\begin{equation}
 \label{eq mvf R1 3}
 \Ra_{1}:  \left \{ \begin{array}{l}
t_1'=\frac{1}{6}\big(t_1^2-t_2^2\big) \\
t_2'=\frac{1}{3}\big(t_1t_2-t_2^2+54t_3\big)\\
t_3'=\frac{2}{3}t_1t_3+\frac{1}{3}t_2t_3+9t_4 \\
t_4'=t_1t_4+t_2t_4 
\end{array} \right. ,
\end{equation}
in which the polynomial relation $t_3^2-t_2t_4=0$ holds.
\begin{enumerate}
  \item A particular solution of $\Ra_1$ is given as follows:
  \begin{equation} \label{eq solution R1 3}
\left \{ \begin{array}{l}
t_1=\too (q):=\frac{1}{4}\big(E_2(q)+3E_2(q^3)\big),\\
t_2=\tot(q):=\frac{1}{2}\big(3E_2(q^3)-E_2(q)\big),\\
t_3=\toth(q):=\eta^8(q^3)+9\frac{\eta^8(q^3)\eta^3(q^9)}{\eta^3(q)},\\
t_4=\tof(q):=\left( \frac{\eta^9(q^3)}{\eta^3(q)} \right)^2,
\end{array} \right.
\end{equation}
with $\tot \in \mfs_2(\Gamma_0(3))$, $\toth \in \mfs_4(\Gamma_0(3))$, $\tof \in \mfs_6(\Gamma_0(3))$ and $\too \in \qmfs_2(\Gamma_0(3))$.
Moreover, if we set
\[
\md_3:=\eta^6({q})\eta^6({q^3}),
\]
then it is a cusp form of weigh $6$ for $\Gamma_0(3)$, $\Delta_3=\tot\toth-27\tof$ and
\[
\md_3'=\too\md_3.
\]
  \item If we consider $\too, \tot, \toth, \tof$ as free parameters and let ${\mathscr{I}}$ and $\tilde{\mathscr{I}}$ to be the ideals
  generated by $\toth^2-\tot\tof$ in $\C[\tot,\toth,\tof]$ and $\C[\too,\tot,\toth,\tof]$, respectively, then:
  \begin{align}
   & \mfs(\Gamma_0(3))\simeq \frac{\C[\tot,\toth,\tof]}{\mathscr{I}}\,, \\
   & \qmfs(\Gamma_0(3))\simeq \frac{\C[\too,\tot,\toth,\tof]}{\tilde{\mathscr{I}}}.
  \end{align}
  \item The vector field $\Ra_1=\frac{1}{6}(t_1^2-t_2^2)\frac{\partial}{\partial t_1}+\frac{1}{3}(t_1t_2-t_2^2+54t_3)\frac{\partial}{\partial t_2}
  +(\frac{2}{3}t_1t_3+\frac{1}{3}t_2t_3+9t_4)\frac{\partial}{\partial t_3}+(t_1t_4+t_3^2)\frac{\partial}{\partial t_4}$  along with the vector fields
  $\rvf_1:=2t_1\frac{\partial}{\partial t_1}+2t_2\frac{\partial}{\partial t_2}+4t_3\frac{\partial}{\partial t_3}+6t_3\frac{\partial}{\partial t_4}$ and
  $\cvf_1:=-6\frac{\partial}{\partial t_1}$ is isomorphic to the Lie algebra $\sl2$.
\end{enumerate}
\end{theo}

We call the systems \eqref{eq mvf R2 3} and \eqref{eq mvf R1 3} the \emph{Ramanujan-type systems} for $\Gamma_0(2)$ and $\Gamma_0(3)$, respectively.

Part {\bf $3$} of Theorem \ref{thm main} and Theorem \ref{thm main2} provides the  spaces $\qmfs\big(\Gamma_0(2)\big)$
and $\qmfs\big(\Gamma_0(3)\big)$ with an $\sl2$-module structure, respectively. This property, in general, is important on the one hand to assign correct weights to
the Calabi-Yau modular forms, see \cite{younes2,younes3}, and on the other hand to study
the dynamics of modular vector fields, see \cite{gui07, guireb}.

Modular forms, in particular their Fourier coefficients, play an important role in number theory. As applications, in Section \ref{section appl}
we use the modular forms and the relations given in the
above theorems to find some interesting and non-trivial recurrence formulas and congruence relations.

The author is aware that some of the proofs and details stated in this paper are maybe standard and known facts for specialists who work on modular forms.
Nevertheless, since he wishes to reach a wider class of readers and researchers in the other areas such as Differential Equations and Holomorphic Dynamical Systems,
he decided to include a section of preliminaries and basic facts, and also write the proofs in a more detailed form.

The structure of the present work is as follows. In Section \ref{section pbf} we state the preliminaries and basic facts for non-experts in
modular form theory, so the experts in this subject can simply skip this section. In Section \ref{section origin} we give the origin of Ramanujan
type systems \eqref{eq mvf R2 3} and \eqref{eq mvf R1 3}. Section \ref{section proof 1} and Section \ref{section proof 2}, respectively, are devoted to the
proofs of Theorem \ref{thm main} and Theorem \ref{thm main2}. In Section \ref{section ma} we find some more analogies for $\qmfs(\Gamma_0(2))$
and $\qmfs(\Gamma_0(3))$ in comparison to $\qmfs(\SL2)$.
Section \ref{section appl} comes up with some applications, where we introduce Ramanujan-type tau functions for $\Gamma_0(2)$ and $\Gamma_0(3)$ and give
some congruence and recurrence relations.

\bigskip

{\bf Acknowledgment.} Some parts of Section \ref{section appl} are results of valuable discussions that the author had during his stay at Max Planck
Institute for Mathematics (MPIM) in Bonn with Pieter Moree. So, he would like to express his sincere gratitude to him and he also wishes to thank MPIM
and its staff for preparing such an excellent ambience for doing mathematical work.

\section{Preliminaries and basic facts } \label{section pbf}
In order to make this paper self contained we start by recalling some basic and useful definitions, terminologies and known facts,
which can be found in any standard reference on modular forms, such as  \cite{bhgz,cs17,dishu}.

Throughout $\Gamma$ refers to a subgroup of $\SL2$ of the finite index which we denote by $[\SL2:\Gamma]$.
Any $\Gamma$ acts from the left on $\H^\ast:=\H\cup \P^1_\Q$ as follows:
\begin{equation}
 \gamma\cdot \tau=\frac{a\tau+b}{c\tau+d},\ \ \forall \tau \in \H^\ast \ {\rm and}  \ \forall \gamma=\left( {\begin{array}{*{20}c}
   {a} & {b}   \\
   { c}  & {d}   \\
\end{array}} \right) \in \Gamma,
\end{equation}
where ${\H}:=\{\tau\in \C \ : \ {\rm Im} (\tau)>0\}$ is the upper half-plane and $\P^1_\Q:=\Q\cup \{\infty\}$. Note that
$\gamma\cdot \infty=\frac{a}{c}$, and $\gamma\cdot \tau=\infty$ if $c\tau+d=0$. It is
well known that $X(\Gamma):=\Gamma\backslash \H^\ast$ is a compact Riemann surface.
We denote by
$\mathscr{C}(\Gamma):=\Gamma\backslash \P^1_\Q$ the set of the cusps of $\Gamma$, by $\overline{\Gamma}$ the image of $\Gamma$ in ${\rm PSL}_2(\Z):=\SL2/\{\pm 1\}$, and by
$d_\Gamma:=[{\rm PSL}_2(\Z):\overline{\Gamma}]$ the index of $\overline{\Gamma}$ in ${\rm PSL}_2(\Z)$. Evidently, if $\Gamma$ is an
\emph{even} subgroup, i.e. $-1\in \Gamma$, then $d_\Gamma=[\SL2:\Gamma]\ast'=q\frac{\partial \ast}{\partial q}=\frac{1}{2\pi i}\frac{d}{d \tau}\ \ and \ \ q=e^{2\pi i\tau}$, and if $\Gamma$ is an
\emph{odd} subgroup, i.e. $-1\notin \Gamma$, then $d_\Gamma=[\SL2:\Gamma]/2$. The \emph{width} $h$ of $\Gamma$ at the cusp $[\infty]$ is the least $h\in \N$ for which
at least one of { \tiny $\left( {\begin{array}{*{20}c}
   {1} & {h}   \\
   { 0}  & {1}   \\
\end{array}} \right)$} and { \tiny $\left( {\begin{array}{*{20}c}
   {-1} & {-h}   \\
   { 0}  & {-1}   \\
\end{array}} \right)$} belongs to $\Gamma$. If $\Gamma$ contains a cusp $P\in \mathscr{C}(\Gamma)$ different from $[\infty]$, then
there exists a $\gamma={ \tiny \left( {\begin{array}{*{20}c}
   {a} & {b}   \\
   { c}  & {d}   \\
\end{array}} \right)}\in \SL2$ such that $[\gamma\cdot \infty]=P$, and the width $h$ of $\Gamma$ at the cusp $P$ is
defined as the width of $\gamma^{-1}\Gamma \gamma$ at the cusp $[\infty]$ (which of course will be independent of the choice of $\gamma$).

For any integer $k$, by a \emph{weakly modular function} of weight $k$ on $\Gamma$ we mean a meromorphic function
$f:{\H} \to \C$ that satisfies the following modularity property:
\begin{equation} \label{eq modular functional}
 (f|_k\gamma)(\tau)=f(\tau),\ \ \forall \ \gamma=\left( {\begin{array}{*{20}c}
   {a} & {b}   \\
   { c}  & {d}   \\
\end{array}} \right) \in \Gamma,
\end{equation}
in which $(f|_k\gamma)(\tau):=(c\tau+d)^{-k} f(\gamma\cdot \tau)$.
Let $f$ be a weakly modular function of weight $k$ on $\Gamma$ and set $q_h:=e^{2\pi i \tau /h}$, where $h$ is the width of $\Gamma$ at $\infty$ and $\tau\in\H$. Then there exists
a meromorphic function $\tilde{f}$ on the punctured disc $\mathbb{D}^\ast:=\{z\in \C \ :\ 0<|z|<1\}$ such that either $f(\tau)=\tilde{f}(q_h)$
or $f(\tau)=q_h^{1/2}\tilde{f}(q_h)$, for all $\tau\in \H$. Hence, if we consider the Laurent expansion of $\tilde{f}$ in $0$, then we get
$f(\tau)=\sum_{n=-\infty}^{+\infty}a_n(f)q_h^n$, or $f(\tau)=\sum_{n=-\infty}^{+\infty}a_n(f)q_h^{n+1/2}$, where $a_n(f)\in \C$ for all
$n\in \Z$. This expansion is known as the \emph{$q$-expansion} of $f$ at $\infty$, and sometimes, by abuse of notation, we write $f(q_h)$ instead $f(\tau)$.
We say that $f$ is meromorphic at $\infty$ if there exists $N\in \Z$ such that $a_N(f)\neq 0$ and $a_n(f)=0$ for all $n<N$. The integer
$N$ is known as the order of $f$ at the cusp $[\infty]$ on $\Gamma$, and we write ${\rm ord}_\infty(f;\Gamma)=N$. Indeed, $f$ is
holomorphic at the cusp $[\infty]$ if ${\rm ord}_\infty(f;\Gamma)\geq 0$ and $f$ vanishes at $[\infty]$ if ${\rm ord}_\infty(f;\Gamma)> 0$.
If $P\in \mathscr{C}(\Gamma)$ is any other cusp of $\Gamma$, then the order of $f$ at the cusp $P$ on $\Gamma$, denoted by
${\rm ord}_P(f;\Gamma)$, is defined as ${\rm ord}_P(f;\Gamma):={\rm ord}_\infty(f|_k\gamma;\gamma^{-1}\Gamma\gamma)$, where $\gamma\in\SL2$ is
the same as above such that $[\gamma\cdot \infty]=P$ (one can see that this definition is independent of the choice of $\gamma$).
Thus, $f$ is meromorphic at $P$ if ${\rm ord}_P(f;\Gamma)\in \Z$, $f$ is holomorphic at $P$ if ${\rm ord}_P(f;\Gamma)\geq 0$,
and $f$ vanishes at $P$ if ${\rm ord}_P(f;\Gamma)>0$. Indeed, if $f$ is holomorphic at $P$, then one can define $f(P):=(f|_k\gamma)(\infty)$
provided $k$ is even, but for odd $k$'s $f(P)$ can be defined up to sign.

A \emph{weakly modular form} of weight $k$ on $\Gamma$ is a weakly modular function $f:\H \to \C$ of weight $k$ on $\Gamma$ that is
holomorphic in $\H$. Moreover, if $f$ is holomorphic in the all cusps of $\Gamma$, then we call it a \emph{modular form} of weight $k$
on $\Gamma$. A modular form $f$ of weight $k$ on $\Gamma$ is called a \emph{cusp form} of weight $k$ on $\Gamma$ if it vanishes
in all cusps of $\Gamma$. We denote the space of weakly modular functions, weakly modular forms, modular forms and cusp forms of
weight $k$ on $\Gamma$,
respectively, by $\mfs^\ast_k(\Gamma)$, $\mfs^!_k(\Gamma)$, $\mfs_k(\Gamma)$ and $\cfs_k(\Gamma)$, which are $\C$-vector spaces.

We state below the valence formula (for more details see \cite[\S 5.6.2 and Theorem 5.6.11]{cs17}). Here
$e_\Gamma(\tau)$ refers to the order of the stabilizer of $\tau$ in $\overline{\Gamma}$. If $e_\Gamma(\tau)>1$, then $\tau$ is called an
\emph{elliptic point} of $\Gamma$ (or $X(\Gamma)$). In fact, $e_\Gamma(\tau)\in \{1,2,3\}$, and if $\tau$ does not belong to the $\SL2$-orbit
of $i:=\sqrt{-1}$ or $e^{2\pi i/3}$, then $e_\Gamma(\tau)=1$.

\begin{theo}{\bf (The valence formula)} \label{thm valence}
Let $f\in \mfs^\ast_k(\Gamma)$ be a non-zero weakly modular function. Then the following holds:
 \begin{equation}
  \sum_{\tau\in\Gamma\backslash \H} \frac{{\rm ord}_\tau (f)}{e_\Gamma(\tau)}+\sum_{P\in \mathscr{C}(\Gamma)}{\rm ord}_P (f;\Gamma)
  =d_\Gamma\frac{k}{12}.
 \end{equation}
\end{theo}

\begin{rem}
 Note that if $f\in\mfs_k(\Gamma)$ is a modular form of
weight $k$, then for any $\tau\in\H$ and any $P\in \mathscr{C}(\Gamma)$, ${\rm ord}_\tau (f)\geq 0$ and ${\rm ord}_P (f;\Gamma)\geq 0$.
In particular, ${\rm ord}_\tau (f)= 0$ or ${\rm ord}_P (f;\Gamma)= 0$ if and only if $f(\tau)\neq 0$ or $f(P)\neq 0$, respectively. Hence,
the valence formula implies that:
 \begin{description}
  \item {\bf (i)} if $k<0$, then $\mfs_k(\Gamma)=0$,
  \item {\bf (ii)} $\cfs_0(\Gamma)=0$, since for non-zero cusp forms $f$ we have,
        $\sum_{P\in \mathscr{C}(\Gamma)}{\rm ord}_P (f;\Gamma)>0$,
  \item {\bf (iii)} $\mfs_0(\Gamma)=\C$, since for any non-constant modular forms
        $f \in \mfs_0(\Gamma)$ we obtain that $(f-f(\infty))\in \mfs_0(\Gamma)$ is a non-zero modular form and
        $\sum_{P\in \mathscr{C}(\Gamma)}{\rm ord}_P (f-f(\infty);\Gamma)>0$.
 \end{description}
Therefore, we can consider the graded algebras (or the spaces) of modular forms
and cusp forms on $\Gamma$, respectively, as $\mfs(\Gamma)=\bigoplus_{k=0}^\infty \mfs_k(\Gamma)$ and
$\cfs(\Gamma)=\bigoplus_{k=0}^\infty \cfs_k(\Gamma)$.
\end{rem}

Another immediate result of the valence formula is the following very useful fact.
\begin{coro} {\bf (The effectiveness of modular forms)} \label{cor eff}
Let $f,g\in\mfs_k(\Gamma)$, for some non-negative integer $k$, having $q$-expansions:
  $f(q)=\sum_{j=0}^\infty a_f(j)q^j$ and $g(q)=\sum_{j=0}^\infty a_g(j)q^j$.
If $a_f(j)=a_g(j)$ for all $0\leq j \leq \lfloor d_\Gamma\tfrac{k}{12} \rfloor$, then $f=g$.
 \end{coro}
 \begin{proof}
  Let $h:=f-g$. Then $h\in \mfs_k(\Gamma)$ and by hypothesis ${\rm ord}_\infty (f;\Gamma)\geq d_\Gamma\frac{k}{12}+1$, hence the valence formula implies that
  $h\equiv 0$.
 \end{proof}

A \emph{quasi-modular form} of weight $k$ and depth $p$ on $\Gamma$ is a holomorphic function $f$ of moderate growth such that for all
$\tau \in \H$ and $\gamma={ \tiny \left( {\begin{array}{*{20}c}
   {a} & {b}   \\
   { c}  & {d}   \\
\end{array}} \right)}\in \Gamma$ we get:
\[
 (f|_k \gamma)(\tau)=\sum_{j=0}^pf_j(\tau)\left(\frac{c}{c\tau+d}\right)^j,
\]
where $f_1,f_2,\ldots,f_p$ are holomorphic functions of moderate growth, $f_0=f$ and $f_p\neq 0$. We denote the space of
quasi-modular forms of weight $k$ on $\Gamma$ by $\qmfs_k(\Gamma)$, and the space (or graded algebra) of quasi-modular forms
on $\Gamma$ by $\qmfs(\Gamma)=\bigoplus_{k=0}^\infty \qmfs_k(\Gamma)$.  For example one can check that
(see for instance \cite{bhgz}) for all $\gamma={ \tiny \left( {\begin{array}{*{20}c}
   {a} & {b}   \\
   { c}  & {d}   \\
\end{array}} \right)}\in \SL2$:
\[
 (E_2|_2\gamma)(\tau)=E_2(\tau)-\frac{6i}{\pi}\frac{c}{c\tau+d},
\]
hence $E_2$ is quasi-modular form of weight $2$ and depth $1$ on $\SL2$. Or, if $f\in \mfs_k(\Gamma)$ for some integer $k\geq 0$, then one
can easily observe that for all $\gamma={\tiny  \left( {\begin{array}{*{20}c}
   {a} & {b}   \\
   { c}  & {d}   \\
\end{array}} \right)}\in \Gamma$:
\begin{equation} \label{eq f'}
 (f'|_{k+2}\gamma)(\tau)=f'(\tau)+\frac{k}{2\pi i}f(\tau)\frac{c}{c\tau+d},
\end{equation}
which implies that $f'$ is a quasi-modular form of weight $k+2$ and depth $1$ on $\Gamma$. In particular, if $k>0$, then $f'$ is not
a modular form, and therefore the space of modular forms is not closed under the differentiation.
\begin{rem} \label{rem qmfs=mfs[lambda]}
 Let $\Gamma$ be a non-compact subgroup of $\SL2$. Due to \cite[Proposition 20]{bhgz}, the space of quasi-modular forms $\qmfs(\Gamma)$
 is closed under differentiation. Furthermore, there exists a $\lambda\in\qmfs_2(\Gamma)\setminus \mfs_2(\Gamma)$ such that any quasi-modular form
 on $\Gamma$ is a polynomial in $\lambda$ with modular coefficients, i.e., $\qmfs(\Gamma)=\mfs(\Gamma)[\lambda]$.
\end{rem}

As we observed above the derivative of a modular form is not necessarily a modular form, but there are certain combinations
of modular forms and their derivatives which are again modular forms. More precisely, for any non negative integer $n$ and
any $f\in \mfs_k(\Gamma),  \  g\in \mfs_l(\Gamma), \ k,l\in \Z_{\geq 0}$, the $n$-th Rankin-Cohen bracket $[f,g]_n$ is defined as:
\begin{equation}\label{eq rcb}
  [f,g]_n:=\sum_{i+j=n}(-1)^j\binom{n+k-1}{i}\binom{n+l-1}{j}f^{(j)}g^{(i)} \, ,
\end{equation}
where $f^{(j)}$ and $g^{(j)}$ refer to the $j$-th derivative of $f$ and $g$ with respect to the derivation $\ast'$ given in
\eqref{eq ramanujan}. Cohen \cite{coh77} proved that:
\begin{equation} \label{eq RCb is mf}
 [f,g]_k\in \mathscr{M}_{k+l+2n}(\Gamma).
\end{equation}
Zagier \cite{zag94} introduced Rankin-Cohen algebraic structures and studied them.

\begin{rem} \label{rem basic}
 Since in this paper we are considering $\Gamma=\Gamma_0(2)$ and $\Gamma=\Gamma_0(3)$, we summarize some important properties and facts about
 $\Gamma_0(N)$ valid for all $N\in \N$. Recall that
 \[
  \Gamma_0(N):=\left\{ \left( {\begin{array}{*{20}c}
   {a} & {b}   \\
   { c}  & {d}   \\
\end{array}} \right) \in \SL2 \ : \ \left( {\begin{array}{*{20}c}
   {a} & {b}   \\
   { c}  & {d}   \\
\end{array}} \right) \equiv \left( {\begin{array}{*{20}c}
   {\ast} & {\ast}   \\
   { 0}  & {\ast}   \\
\end{array}} \right) \ ({\rm mod} \ N) \right\},
 \]
is an even congruence subgroup of $\SL2$ of finite index. We have the following facts, in which we suppose that $N\in\N$ and $p$ is a prime
number.
\begin{description}
 \item {\bf (i)} If $k$ is odd, then $\mfs_k(\Gamma_0(N))=0$.
 \item {\bf (ii)} If $f(\tau)\in \mfs_k\big(\SL2\big)$ or $f(\tau)\in\cfs_k\big(\SL2\big)$, then $f(N\tau)\in\mfs_k\big(\Gamma_0(N)\big)$ or
 $f(N\tau)\in\cfs_k\big(\Gamma_0(N)\big)$, respectively.
 \item {\bf (iii)} $d_{\Gamma_0(p)}=[\SL2:\Gamma_0(p)]=p+1$.
 \item {\bf (iv)} $\mathscr{C}(\Gamma_0(p))=\{[\infty],[0]\}$, and the width of $\Gamma_0(p)$ at $[\infty]$ and $[0]$, respectively,
is $1$ and $p$.
 \item {\bf (v)} If we denote the number of elliptic points of period $2$ in $X_0(p):=\Gamma_0(p)\backslash \H^\ast$ by $\varepsilon_2$,
then $\varepsilon_2=2$ if $p \equiv 1\ ({\rm mod}\ 4)$, $\varepsilon_2=0$ if $p \equiv 3\ ({\rm mod} \ 4)$, and $\varepsilon_2=1$ if $p =2$.
 \item {\bf (vi)} If we denote the number of elliptic points of period $3$ in $X_0(p)$ by $\varepsilon_3$, then $\varepsilon_3=2$ if
$p \equiv 1\ ({\rm mod}\ 3)$, $\varepsilon_3=0$ if $p \equiv 2\ ({\rm mod} \ 3)$, and $\varepsilon_3=1$ if $p =3$.
 \item {\bf (vii)} If we denote by $g$ the genus of $X_0(p)$, then $g=\lfloor \frac{p+1}{12}\rfloor-1$ if $p+1 \equiv 2\ ({\rm mod} \ 12)$,
and $g= \lfloor \frac{p+1}{12}\rfloor$ otherwise.
 \item {\bf (viii)} If $k\geq 2$ is an even integer, then $\dim \mfs_k(\Gamma_0(p))=(k-1)(g-1)+\lfloor \frac{k}{4} \rfloor \varepsilon_2+
\lfloor \frac{k}{3} \rfloor \varepsilon_3+k$.
 \item {\bf (ix)} If $k\geq 4$ is an even integer, then $\dim \cfs_k(\Gamma_0(p))=(k-1)(g-1)+\lfloor \frac{k}{4} \rfloor \varepsilon_2+
\lfloor \frac{k}{3} \rfloor \varepsilon_3+k-2$, and $\dim \cfs_2(\Gamma_0(p))=g$.
\end{description}
For more details see, for example, \cite{dishu}.
\end{rem}

\begin{rem} \label{rem eta quotient}
It is well known that $\eta(q)$ is a modular form of weight $1/2$ (and non-trivial character) on $\SL2$ which does not vanish on $\H$.
Due to \eqref{eq Delta/Delta'} we know that
$\eta'(q) / \eta(q)=\frac{1}{24}E_2(q)$, which is a quasi-modular form of weight $2$ on $\SL2$.
If for $g=r_1^{t_1}r_2^{t_2}\ldots r_s^{t_s}, \, r_j> 0 , \, t_j\in \Z,$ we suppose that the $\eta$-quotient $\eta_g(q):=\prod_{j=1}^{s}\eta^{t_j}(q^{r_j})$ is a modular form (modular function) on $\Gamma$, then:
\begin{equation} \label{eq eta quotient}
\frac{\eta_g'(q)}{\eta_g(q)}=\left(\log (\eta_g(q))\right)'=\sum_{j=1}^{s} \frac{r_jt_j}{24}E_2(q^{r_j}).
\end{equation}
It is a quasi-modular form of weight $2$ on $\Gamma$ (note that $\eta_g$ does not vanish on $\H$,
and, $\eta_g$ and $\eta'_g$ have the same vanishing order at $\infty$). Moreover, if $\eta_g$ is a modular function (i.e., weakly
modular function of weight $0$), then using \eqref{eq f'} we conclude that $\eta_g'(q) / \eta_g(q)$ is a modular form of weight $2$ on $\Gamma$.
\end{rem}
\section{Origin of the systems $\Ra_1$ and $\Ra_2$ } \label{section origin}

The initial version of the systems \eqref{eq mvf R2 3} and  \eqref{eq mvf R1 3} first appeared in  \cite{movnik},
where by applying an algebraic method, called \emph{Gauss-Manin connection in disguise} (GMCD),  in a geometric context the author and Hossein Movasati
found these systems as a unique vector field on a special moduli space that satisfies certain conditions. More precisely,
for any positive integer $n$ we obtained a one-parameter family $X:=X_z$, $z\in \P^1\setminus \{0,1,\infty\}$, of Calabi-Yau
$n$-folds arising from the Dwork family:
\[
W_z:=\big\{zx_0^{n+2}+x_1^{n+2}+x_2^{n+2}+\cdots+x_{n+1}^{n+2}-(n+2)x_0
x_1x_2\cdots x_{n+1}=0\big\}\subset \P^{n+1},
\]
and we introduced the moduli space $\T$ of the pairs $(X,[\alpha_1,\alpha_2,\ldots,\alpha_{n+1}])$,
in which  $\{\alpha_1,\alpha_2,\ldots,\alpha_{n+1}\}$ is a basis of the $n$-th algebraic de Rham cohomology $H^n_{\rm dR}(X)$ satisfying some
specific properties. In the main theorem of the same work \cite{movnik} we proved that there exist a unique vector field $\Ra:=\Ra_n$ and
regular functions $\Yuk_j, \ j=1,2,\ldots,n-2$ in $\T$
such that the Gauss-Manin connection of the universal family of $\T$ composed with the vector field $\Ra$,
namely $\nabla_{\Ra}$, satisfies:
\begin{equation}
\label{jimbryan}
\nabla_{\Ra}
\begin{pmatrix}
\alpha_1\\
\alpha_2 \\
\alpha_3 \\
\vdots \\
\alpha_n \\
\alpha_{n+1} \\
\end{pmatrix}
= \underbrace {\begin{pmatrix}
0& 1 & 0&0&\cdots &0&0\\
0&0& \Yuk_1&0&\cdots   &0&0\\
0&0&0& \Yuk_2&\cdots   &0&0\\
\vdots&\vdots&\vdots&\vdots&\ddots   &\vdots&\vdots\\
0&0&0&0&\cdots   &\Yuk_{n-2}&0\\
0&0&0&0&\cdots   &0&-1\\
0&0&0&0&\cdots   &0&0\\
\end{pmatrix}}_\Yuk
\begin{pmatrix}
\alpha_1\\
\alpha_2 \\
\alpha_3 \\
\vdots \\
\alpha_n \\
\alpha_{n+1} \\
\end{pmatrix}.
\end{equation}
Indeed, we get $\T=\spec\big(\C[t_1,t_2,\ldots, t_{\dt},\frac{1}{t_{n+2}(t_{n+2}-t_1^{n+2}) }]\big)$, where:
\begin{equation}
\label{29dec2015} \dt=\dt_n=\dim \T=\left \{
\begin{array}{l}
\frac{(n+1)(n+3)}{4}+1,\,\, \quad  \textrm{\rm if \textit{n} is odd};
\\\\
\frac{n(n+2)}{4}+1,\,\,\,\,\ \quad\quad \textrm{\rm if \textit{n} is
even}.
\end{array} \right.
\end{equation}
In particular, for $n=1$ we computed the vector field $\Ra_1$ explicitly as follows:
\[
\Ra_1=\left(-t_1t_2-9(t_1^3-t_3)\right)\frac{\partial}{\partial t_1}+\left(81t_1(t_1^3-t_3)-t_2^2\right)\frac{\partial}{\partial t_2}
+\left( -3t_2t_3 \right)\frac{\partial}{,\partial t_3},
\]
and for $n=2$ we got:
\[
\Ra_2=\left(t_3-t_1t_2\right)\frac{\partial}{\partial t_1}+\left(2t_1^2-\frac{1}{2}t_2^2\right)\frac{\partial}{\partial t_2}
+\left(-2t_2t_3+8t_1^3\right)\frac{\partial}{\partial t_3}+\left(-4t_2t_4\right)\frac{\partial}{\partial t_4},
\]
where for $\Ra_2$ the following polynomial equation holds among the $t_i$'s
\begin{equation}\label{eq t3}
 t_3^2=4(t_1^4-t_4).
\end{equation}
We explain below how to get the systems \eqref{eq mvf R2 3} and \eqref{eq mvf R1 3} from $\Ra_2$ and $\Ra_1$, respectively.

For $n=2$ let us consider the following representation of $\Ra_2$ as a system of ODEs:
\begin{equation}
 \label{eq mvf R2}
 \Ra_{2}:\left \{ \begin{array}{l}
\dot{t}_1=t_3-t_1t_2\\
\dot{t}_2=2t_1^2-\frac{1}{2}t_2^2\\
\dot{t}_3=-2t_2t_3+8t_1^3\\
\dot{t}_4=-4t_2t_4
\end{array} \right.,
\end{equation}
where $\dot \ast=a\cdot q\cdot \frac{\partial
\ast}{\partial q}$ refers to a logarithmic derivation with some constant $a\in \C$. In \cite{movnik} we computed the $q$-expansion of a solution of
\eqref{eq mvf R2} for $a=-\frac{1}{5}$ by computer and we observed that at least the first 100 coefficients of the $q$-expansions coincide with the coefficients of the following functions:
\begin{equation} \label{eq solution R02}
\left \{ \begin{array}{l}
\frac{10}{6}{t}_1\big(\frac{q}{10}\big)=\frac{1}{24}\big(\theta_3^4(q^2)+\theta_2^4(q^2)\big),\\
\frac{10}{4}{t}_2\big(\frac{q}{10}\big)=\frac{1}{24}\big(E_2(q^2)+2E_2(q^4)\big),\\
10^4{t}_4\big(\frac{{q}}{10}\big)=\eta^8({q})\eta^8({q}^2),
\end{array} \right.
\end{equation}
but we did not prove theoretically that these functions form a solution of the system \eqref{eq mvf R2}.
If we apply the change of variables $\tilde{t}_1=20t_2$, $\tilde{t}_2=40t_1$ and $\tilde{t}_3=800t_3$, then the system \eqref{eq mvf R2} transforms
to the system:
\begin{equation}
 \label{eq mvf R2 4}
 \left \{ \begin{array}{l}
\tilde{t}_1'=\frac{1}{8}(\tilde{t}_1^2-\tilde{t}_2^2) \\
\tilde{t}_2'=\frac{1}{4}(\tilde{t}_1\tilde{t}_2-\tilde{t}_3)\\
\tilde{t}_3'=\frac{1}{2}(\tilde{t}_1\tilde{t}_3-\tilde{t}_2^3)
\end{array} \right. ,
\end{equation}
in which $\ast'= q \frac{\partial \ast}{\partial q}=-5\dot \ast$. Note that due to \eqref{eq t3} $t_4$ depends to $t_1$ and
$t_3$, hence we can omit it, indeed, in Theorem \ref{thm main} we consider $\Delta_2=t_4$.
If, by abuse of notation, we use again $t_1,t_2,t_3$ instead of $\tilde{t}_1,\tilde{t}_2,\tilde{t}_3$, then we get the system \eqref{eq mvf R2 3} from
\eqref{eq mvf R2 4}.

For $n=1$ we have:
\begin{equation} \label{eq mvf R1}
 \Ra_{1}:  \left \{ \begin{array}{l} \dot
t_1=-t_1t_2-9(t_1^3-t_3)
\\
\dot t_2=81t_1(t_1^3-t_3)-t_2^2
\\
\dot t_3=-3t_2t_3
\end{array} \right. ,
\end{equation}
where $\dot \ast=3\cdot q\cdot \frac{\partial \ast}{\partial q}$. We verified that at least the first $100$ coefficients of the $q$-expansions of the
following quasi-modular forms satisfy the system \eqref{eq mvf R1}:
\begin{equation} \label{eq solution R01}
\left \{ \begin{array}{l}
{t}_1(q)=\frac{1}{3}\big(2\theta_3(q^2)\theta_3(q^6)-\theta_3(-q^2)\theta_3(-q^6)\big),\\
{t}_2(q)=\frac{1}{8}\big(E_2(q^2)-9E_2(q^6)\big),\\
{t}_3(q)=\frac{\eta^9({q}^3)}{\eta^3({q})}.
\end{array} \right.
\end{equation}
Here $t_1, t_3$ are modular forms of weight $1,3$, respectively, and character $\chi_{-3}$  for $\Gamma_0(3)$, where
$\chi_{-3}(d):=\left( \frac{d}{3} \right), \ d\in \Z$ (here $\left(\frac{\cdot}{3}\right)$ refers to the Jacobi (Legendre) symbol). Note that neither 
of the eta quotients
$t_3(q)=\frac{\eta^9({q}^3)}{\eta^3({q})}$ and $27(t_1(q)^3-t_3(q))=\frac{\eta^9({q})}{\eta^3({q}^3)}$ are cusp forms for $\Gamma_0(3)$, while
$27t_3(q)(t_1(q)^3-t_3(q))=\eta^6({q})\eta^6({q}^3)=\Delta_3$ is a cusp form of weight $6$ for $\Gamma_0(6)$ (see Section \ref{section proof 2}).
By applying the change of variables $\tilde{t}_1=-2t_2-9t_1^2$, $\tilde{t}_2=9t_1^2$, $\tilde{t}_3=3t_1t_3$ and $\tilde{t}_4=t_3^2$ to the system
\eqref{eq mvf R1}, we get the system \eqref{eq mvf R1 3}, where again by abuse of notation we use $t_1,t_2,t_3,t_4$ instead of
$\tilde{t}_1,\tilde{t}_2,\tilde{t}_3, \tilde{t}_4$, and $\ast'= q \frac{\partial \ast}{\partial q}=\frac{1}{3}\dot \ast$.

It is worth to point out that the Ramanujan system also can be reencountered through GMCD. Indeed, Hossein Movasati \cite{ho14}
showed that the vector fields $\Ramvf$ satisfies
{ $
\nabla_\Ramvf\alpha=\left(
                    \begin{array}{cc}
                      0 & 1 \\
                      0 & 0 \\
                    \end{array}
                  \right) \alpha \ ,
$ }
where $\alpha$ is the transpose of the vector $(\ [\frac{dx}{y}]\ \ [\frac{xdx}{y}]\ )$ and $\nabla$ is the Gauss-Manin connection of the universal family
of the elliptic curves:
\begin{equation} \label{eq Weierstrass family}
 y^2=4(x-t_1)^3-t_2(x-t_1)-t_3 \, ,  \ \text{where} \ \ (t_1,t_2,t_3)\in\C^3  \ \text{such that} \ \  27t_3^2-t_2^3\neq 0\ .
\end{equation}
In the same work one can find that the modular discriminant $\md$ can be written in terms of the parameters of the family \eqref{eq Weierstrass family}
which is actually a constant multiple of the discriminant of this family, i.e., $27t_3^2-t_2^3$. Analogously, we find out that the modular forms
$\md_2$ and $\md_3$ are a factor of the discriminant of a modified version of the Dwork family. More precisely, if for any integer $n$ we let
$z=\frac{t_{n+2}}{t_1^{n+2}}$, then the Dwork family is equivalent to the family:
\[
 t_{n+2}x_0^{n+2}+x_1^{n+2}+x_2^{n+2}+\cdots+x_{n+1}^{n+2}-(n+2)t_1x_0
x_1x_2\cdots x_{n+1}=0,
\]
whose discriminant is $t_{n+2}(t_1^{n+2}-t_{n+2})$ (see \cite{movnik} for more details). As we saw above, for $n=1$ we get
$\md_3=27t_3(t_1^3-t_3)$, and for $n=2$ we have $\md_2=t_4$.

\section{Proof of Theorem \ref{thm main}} \label{section proof 1}
First note that, due to Remark \ref{rem basic},  $\Gamma_0(2)$  has two cusps $[0],[\infty]$ and one elliptic point of period $2$, and the genus of $X_0(2)$ is zero. Hence:
\begin{align}
  & \dim \mfs_k\big(\Gamma_0(2)\big)=\left\lfloor \frac{k}{4} \right\rfloor +1\,, \ \ {\rm provided} \ k\geq2 \ {\rm is \ even},  \label{eq dim MkG0(2)}\\
  & \dim \cfs_k\big(\Gamma_0(2)\big)=\left\lfloor \frac{k}{4} \right\rfloor -1\,, \ \ {\rm provided} \ k\geq4 \ {\rm is \ even}, \label{eq dim SkG0(2)}
\end{align}
otherwise $\mfs_k\big(\Gamma_0(2)\big)=0$ and $\cfs_k\big(\Gamma_0(2)\big)=0$. On account of Remark \ref{rem basic} {\bf (ii)} one gets that $\tto,\ttt,\ttth$
(given in \eqref{eq solution R2 3}) are (quasi-)modular forms for $\Gamma_0(2)$. Their $q$-expansions are as follows:
\begin{align}
  \label{eq q-exp t1}  \tto&= 1-8q-40q^2-32q^3-104q^4-48q^5-160q^6-64q^7- \ldots \,,  \\
  \label{eq q-exp t2}  \ttt&= 1+24q+24q^2+96q^3+24q^4+144q^5+96q^6+192q^7+\ldots  \,, \\
  \label{eq q-exp t3}  \ttth&=1-80q-400q^2-2240q^3-2960q^4-10080q^5-11200q^6-27520q^7-\ldots\,.
\end{align}
If we denote the vector space generated by functions $f_1,f_2,\ldots,f_n$ over $\C$ by $\langle f_1,f_2,\ldots,f_n\rangle$, then we claim that:
\begin{align}\label{eq generator M2}
   & \mfs_2(\Gamma_0(2))= \langle \ttt \rangle,\\
   & \mfs_4(\Gamma_0(2))= \langle \ttt^2,\ttth \rangle , \label{eq generator M4}\\
   & \mfs_6(\Gamma_0(2))= \langle \ttt^3,\ttt\ttth \rangle, \label{eq generator M6}\\
   & \mfs_8(\Gamma_0(2))= \langle \ttt^4,\ttt^2\ttth,\ttth^2 \rangle. \label{eq generator M8}
\end{align}
Indeed, due to Remark \ref{rem eta quotient} we obtain:
\begin{align}\label{eq t2}
   & \frac{\left(\eta^{24}(q^2)/ \eta^{24}(q)\right)'}{\left(\eta^{24}(q^2)/ \eta^{24}(q)\right)}=2E_2(q^2)-E_2(q)=\ttt,
\end{align}
where $\eta^{24}(q^2)/ \eta^{24}(q)$ is a modular function (of weight $0$) for $\Gamma_0(2)$.
Hence, $\ttt\in \mfs_2(\Gamma_0(2))$, and since $\dim \mfs_2(\Gamma_0(2))=1$, \eqref{eq generator M2} is proved.
It is evident that $\ttth\in \mfs_4(\Gamma_0(2))$, and since $\dim \mfs_4(\Gamma_0(2))=2$, in order to prove \eqref{eq generator M4} it is enough to
show that $\ttt^2$ and $\ttth$ are linearly independent on $\C$. To this end, due to the effectiveness of modular forms given in Corollary
\ref{cor eff}, it suffices to observe that the vectors $\left(a_{\ttt^2}(0),a_{\ttt^2}(1)\right)=(1,48)$ and
$\left(a_{\ttth}(0),a_{\ttth}(1)\right)=(1,-80)$  are linearly independent (note that $d_{\Gamma_0(2)}=3$). Analogously,  \eqref{eq generator M6}
follows from the fact that $\left(a_{\ttt^3}(0),a_{\ttt^3}(1)\right)=(1,72)$ and $\left(a_{\ttt\ttth}(0),a_{\ttt\ttth}(1)\right)=(1,-56)$ are
linearly independent, and the linearly independence of  $\left(a_{\ttt^4}(0),a_{\ttt^4}(1),a_{\ttt^4}(2)\right)=(1,96,3552)$,
$\left(a_{\ttt^2\ttth}(0),a_{\ttt^2\ttth}(1),a_{\ttt^2\ttth}(2)\right)=(1,-32,-3616)$ and
$\left(a_{\ttth^2}(0),a_{\ttth^2}(1),a_{\ttth^2}(2)\right)=(1,-160,$ $5600)$ imply \eqref{eq generator M8}.
\begin{enumerate}
  \item By Remark \ref{rem eta quotient} we know that $\Delta_2(q)=\eta^8(q)\eta^8(q^2)\in \mfs_8\big(\Gamma_0(2)\big)$,
  which does not vanish in $\H$, hence on account of the valence formula given in Theorem \ref{thm valence}
  ${\rm ord}_{[0]} \big(\Delta_2;\Gamma_0(2)\big)+{\rm ord}_{[\infty]} \big(\Delta_2;\Gamma_0(2)\big)=2$. Note that
      the $q$-expansion of $\Delta_2$ at $\infty$ is as follows:
      \[
      \Delta_2(q)={q}\prod_{k=1}^\infty (1-{q}^{k})^8(1-{q}^{2k})^8=q-8q^2+12q^3+64q^4-210q^5-96q^6+\ldots,
      \]
      thus ${\rm ord}_{[\infty]} \big(\Delta_2;\Gamma_0(2)\big)=1$, which implies
      \begin{equation} \label{eq ord_0 Delta2}
       {\rm ord}_{[0]} \big(\Delta_2;\Gamma_0(2)\big)=1.
      \end{equation}
  Therefore $\Delta_2\in \cfs_8\big(\Gamma_0(2)\big)$ and \eqref{eq eta quotient} gives:
  \[
  \frac{\Delta_2'(q)}{\Delta_2(q)}=\frac{1}{3}\big(E_2(q)+2E_2(q^2)\big)=\tto.
  \]
  To prove $\tto'=\frac{1}{8}\big(\tto^2-\ttt^2\big)$ and $\ttt'=\frac{1}{4}\big(\tto\ttt-\ttth\big)$ we first point out that:
  \begin{align}\label{}
    & \tto'-\frac{1}{8}\tto^2= \left(\frac{\Delta_2'}{\Delta_2}\right)'-\frac{1}{8}\frac{\Delta_2'^2}{\Delta_2^2}=
\frac{8\Delta_2''\Delta_2-9\Delta'^2}{8\Delta_2^2}=\frac{[\Delta_2,\Delta_2]_2}{72\Delta_2^2},\\
    & \ttt'-\frac{1}{4}\tto\ttt= t'_2-\frac{2\ttt\Delta'_2}{8\Delta_2}=\frac{8\ttt'\Delta_2-2\ttt\Delta_2'}{8\Delta_2}=\frac{[\Delta_2,\ttt]_1}{8\Delta_2},
  \end{align}
 where $[\cdot,\cdot]_1$ and $[\cdot,\cdot]_2$ are the first and second Rankin-Cohen brackets. Hence, due to \eqref{eq RCb is mf}, both
 $\tto'-\frac{1}{8}\tto^2$ and $\ttt'-\frac{1}{4}\tto\ttt$ belong to $\mfs_4\big(\Gamma_0(2)\big)$. Therefore, after comparing
 the coefficients of the $q$-expansion of $\tto'-\frac{1}{8}\tto^2$ and $\ttt'-\frac{1}{4}\tto\ttt$ with $\ttt^2$ and $\ttth$, from
 \eqref{eq generator M4} and the effectiveness of modular forms we get $\tto'-\frac{1}{8}\tto^2=-\frac{1}{8}\ttt^2$ and
 $\ttt'-\frac{1}{4}\tto\ttt=-\frac{1}{4}\ttth$. Analogously, we observe that:
 \[
   \ttth'-\frac{1}{2}\tto\ttth=\frac{[\Delta_2,\ttth]_1}{8\Delta_2}\in \mfs_6\big(\Gamma_0(2)\big),
 \]
 and \eqref{eq generator M6} implies that $\ttth'-\frac{1}{2}\tto\ttth=-\frac{1}{2}\ttt^3$, and this finishes the proof of
 $\ttth'=\frac{1}{2}\big(\tto\ttth-\ttt^3\big)$.
 Finally, using \eqref{eq generator M8} and the effectiveness of modular forms we get that:
 \begin{equation} \label{eq del=t2-t3}
  \Delta_2=\frac{1}{256}\big(\ttt^4-\ttth^2\big),
 \end{equation}
 and this completes the proof of part 1.

  \item To prove $\mfs\big(\Gamma_0(2)\big)=\C[\ttt,\ttth]$ it is enough to show that for any non-negative integer $k$, the set:
  \[
   {\mathscr{B}}_k:=\big\{\ttt^r \ttth^s\ : \ r,s\in \Z_{\geq 0} \ {\rm and} \ 2r+4s=2k \big\},
  \]
 forms a basis for $\mfs_{2k}\big(\Gamma_0(2)\big)$ (note that $\mfs_{2k+1}\big(\Gamma_0(2)\big)=0$). To this end, first note that:
 \begin{align*}
  \ttt^r \ttth^s\in \mathscr{B}_k &\Longleftrightarrow r,s\in \Z_{\geq 0} \ {\rm and} \ 2r+4s=2k ,\\
   &\Longleftrightarrow   r=k-2s \ {\rm and} \ 0\leq s \leq \frac{k}{2}, \ s\in  \Z_{\geq 0},
 \end{align*}
which implies:
  \[
    \# \mathscr{B}_k=\left\lfloor \frac{k}{2} \right\rfloor+1=\dim \mfs_{2k}\big(\Gamma_0(2)\big).
  \]
Hence, it remains to show that the elements of $\mathscr{B}_k$ are linearly independent. We claim that:
\begin{equation} \label{eq different zeros}
 \exists\, P_1,P_2\in \H\cup \{0\} \ {\rm with} \ P_1\neq P_2 \ {\rm such \ that} \ \ttt(P_1)=\ttth(P_2)=0,
\end{equation}
which implies for any two pairs $(r_1,s_1),(r_2,s_2)\in \Z_{\geq 0}\times \Z_{\geq 0}$ with $(r_1,s_1)\neq(r_2,s_2)$, the modular forms
$\ttt^{r_1} \ttth^{s_1}$ and $\ttt^{r_2} \ttth^{s_2}$ have zeros of different order in $\H\cup \{0\}$, and hence they are linearly independent
on $\C$. Then, $\mathscr{B}_k$ is a set of linearly independent elements, and in particular, we proved that
$\ttt$ and $\ttth$ are algebraically independent over $\C$. Thus we have proved that $\mfs\big(\Gamma_0(2)\big)=\C[\ttt,\ttth]$. Now, since $\tto$
is a quasi-modular form (which is not a modular form), Remark \ref{rem qmfs=mfs[lambda]} implies that $\qmfs\big(\Gamma_0(2)\big)=\mfs\big(\Gamma_0(2)\big)
[\tto]=\C[\tto,\ttt,\ttth]$, and this finishes the proof of part 2.

To prove \eqref{eq different zeros}, first note that on account
of \eqref{eq q-exp t2} and \eqref{eq q-exp t3} we have ${\rm ord}_{[\infty]} \big(\ttt;\Gamma_0(2)\big)={\rm ord}_{[\infty]} \big(\ttth;\Gamma_0(2)\big)=0$,
and hence the valence formula gives:
\begin{align}
 &\sum_{\tau\in\Gamma_0(2)\backslash \H} \frac{{\rm ord}_\tau (\ttt)}{e_{\Gamma_0(2)}(\tau)}+{\rm ord}_{[0]} \big(\ttt;\Gamma_0(2)\big)
  =\frac{1}{2}, \label{eq t2 valence}\\
 &\sum_{\tau\in\Gamma_0(2)\backslash \H} \frac{{\rm ord}_\tau (\ttth)}{e_{\Gamma_0(2)}(\tau)}+{\rm ord}_{[0]} \big(\ttth;\Gamma_0(2)\big)
  =1, \label{eq t3 valence}
\end{align}
which imply that there are $P_1,P_2\in\H\cup\{0\}$ such that
$\ttt(P_1)=\ttth(P_2)=0$. If $P_1,P_2\in\H$, then due to \eqref{eq del=t2-t3} and the fact that $\Delta_2$ is non-zero in $\H$
we get that $P_1$ and $P_2$ must be distinct. If, by contradiction, we suppose that $\ttt$ and $\ttth$ do not have
distinct zeros, then the only possibility for $P_1$ and $P_2$ is $P_1=P_2=0$. Thus, \eqref{eq t2 valence} and \eqref{eq t3 valence} imply
that ${\rm ord}_{[0]} \big(\ttt;\Gamma_0(2)\big)=\frac{1}{2}$ and ${\rm ord}_{[0]} \big(\ttth;\Gamma_0(2)\big)=1$, hence
${\rm ord}_{[0]} \big(\ttt^4;\Gamma_0(2)\big)={\rm ord}_{[0]} \big(\ttth^2;\Gamma_0(2)\big)=2$, and therefore
${\rm ord}_{[0]} \big(\Delta_2;\Gamma_0(2)\big)={\rm ord}_{[0]} \big(\frac{1}{256}\big(\ttt^4-\ttth^2\big);\Gamma_0(2)\big)\geq 2$, which contradicts
\eqref{eq ord_0 Delta2}.

  \item Recall that the special linear Lie algebra $\sl2$ is the  Lie algebra of $2\times 2$ matrices with trace zero.
Three matrices
\begin{equation}\label{eq sbsl2}
e:=\left(
        \begin{array}{cc}
          0 & 1 \\
          0 & 0 \\
        \end{array}
      \right)\ \ , \ \ \ \
f:=\left(
        \begin{array}{cc}
          0 & 0 \\
          1 & 0 \\
        \end{array}
      \right)\ \ , \ \ \ \
h:=\left(
        \begin{array}{cc}
          1 & 0 \\
          0 & -1 \\
        \end{array}
       \right)\ ,
\end{equation}
form the standard basis of $\sl2$ with the commutators:
\begin{equation}\label{eq lbsl2}
[e,f]=h \ \ , \ \ \ \ [h,e]=2e \ \ , \ \ \ \ [h,f]=-2f\, .
\end{equation}
We also recall that if we have two vector fields $V=\sum_{j=1}^{\dt}V^j\frac{\partial}{\partial t_j}$ and $W=\sum_{j=1}^{\dt}W^j\frac{\partial}{\partial t_j}$, then
\begin{equation}\label{eq lie bracket in a chart}
[V,W]=VW-WV=\sum_{j=1}^{\dt}\left( V(W^j)-W(V^j) \right) \frac{\partial}{\partial t_j}\,.
\end{equation}
 Using \eqref{eq lie bracket in a chart} one can easily check that:
 \[
  [\Ra_2,\cvf]=\rvf \ , \ \ [\rvf,\Ra_2]=2\Ra_2 \ , \ \ [\rvf,\cvf]=-2\cvf\,.
 \]
Hence, the correspondences $\Ra_2\mapsto e$, $\rvf\mapsto h$ and $\cvf\mapsto f$ show that the Lie algebra generated by $\Ra_2, \ \rvf$ and $\cvf$
is isomorphic to $\mathfrak{sl}_2(\C)$, and this completes the proof of part 3.

\end{enumerate}

\section{Proof of Theorem \ref{thm main2}} \label{section proof 2}
The proof  is analogous to the proof of Theorem \ref{thm main}. Using Remark \ref{rem basic} we get:
\begin{align}
  & \dim \mfs_k\big(\Gamma_0(3)\big)=\left\lfloor \frac{k}{3} \right\rfloor +1\,, \ \ {\rm provided} \ k\geq2 \ {\rm is \ even},  \label{eq dim MkG0(3)}\\
  & \dim \cfs_k\big(\Gamma_0(3)\big)=\left\lfloor \frac{k}{3} \right\rfloor -1\,, \ \ {\rm provided} \ k\geq4 \ {\rm is \ even}, \label{eq dim SkG0(3)}
\end{align}
otherwise $\mfs_k\big(\Gamma_0(3)\big)=0$ and $\cfs_k\big(\Gamma_0(3)\big)=0$. The $q$-expansions of $\too, \ \tot, \ \toth$ and $\tof$ are as follows:
\begin{align}
  \label{eq q-exp t11}  \too&= 1-6q-18q^2-42q^3-42q^4-36q^5-126q^6-48q^7-90q^8-150q^9- \ldots,   \\
  \label{eq q-exp t12}  \tot&= 1+12q+36q^2+12q^3+84q^4+72q^5+36q^6+96q^7+180q^8+12q^9+\ldots ,  \\
  \label{eq q-exp t13}  \toth&=q+9q^2+27q^3+73q^4+126q^5+243q^6+344q^7+585q^8+729q^9+\ldots , \\
  \label{eq q-exp t14}  \tof&=q^2+6q^3+27q^4+80q^5+207q^6+432q^7+863q^8+1512q^9+\ldots ,
\end{align}
and $\too\in\qmfs_2\big(\Gamma_0(3)\big)\setminus \mfs_2\big(\Gamma_0(3)\big)$, $\tot\in\mfs_2\big(\Gamma_0(3)\big)$, $\toth\in\mfs_4\big(\Gamma_0(3)\big)$ and
$\tof\in\mfs_6\big(\Gamma_0(3)\big)$. Note that:
\begin{align}\label{eq t2}
   & \frac{\big(\eta^{24}(q^3)/ \eta^{24}(q)\big)'}{\big(\eta^{24}(q^3)/ \eta^{24}(q)\big)}=3E_2(q^3)-E_2(q)=2\tot,
\end{align}
where $\eta^{24}(q^3)/ \eta^{24}(q)$ is a modular function (of weight $0$) for $\Gamma_0(3)$.
Similarly to \eqref{eq generator M2}-\eqref{eq generator M8} we observe that:
\begin{align}\label{eq generator M23}
   & \mfs_2\big(\Gamma_0(2)\big)= \langle \tot \rangle,\\
   & \mfs_4\big(\Gamma_0(2)\big)= \langle \tot^2,\toth \rangle , \label{eq generator M43}\\
   & \mfs_6\big(\Gamma_0(2)\big)= \langle \tot^3,\tot\toth,\tof \rangle, \label{eq generator M63}\\
   & \mfs_8\big(\Gamma_0(2)\big)= \langle \tot^4,\tot^2\toth,\toth^2 \rangle. \label{eq generator M83}
\end{align}

\begin{enumerate}
  \item We first verify that $\Delta_3(q)=\eta^6(q)\eta^6(q^3)\in \mfs_6\big(\Gamma_0(3)\big)$
  does not vanish in $\H$, furthermore that  ${\rm ord}_{[0]} \big(\Delta_3;\Gamma_0(3)\big)+{\rm ord}_{[\infty]} \big(\Delta_3;\Gamma_0(3)\big)=2$ and
      the $q$-expansion of $\Delta_3$ at $\infty$ is as follows:
      \[
      \Delta_3(q)={q}\prod_{k=1}^\infty (1-{q}^{k})^6(1-{q}^{3k})^6=q-6q^2+9q^3+4q^4+6q^5-54q^6-40q^7+\ldots.
      \]
  Using these facts, we can complete the proof in the same way as the proof part 1 of Theorem \ref{thm main}.

  \item By \eqref{eq q-exp t12}, \eqref{eq q-exp t13} and \eqref{eq q-exp t14} we have ${\rm ord}_{[\infty]} \big(\tot;\Gamma_0(3)\big)=0$,
  ${\rm ord}_{[\infty]} \big(\toth;\Gamma_0(3)\big)=1$
  and ${\rm ord}_{[\infty]} \big(\tof;\Gamma_0(3)\big)=2$. We fix a positive even integer $k$. Let $r_2,r_3,r_4$ and $s_2,s_3,s_4$ be non-negative integers
  such that $2r_2+4r_3+6r_4=k$ and $2s_2+4s_3+6s_4=k$ and set $f_1:=\tot^{r_2}\toth^{r_3}\tof^{r_4}$ and $f_2:=\tot^{s_2}\toth^{s_3}\tof^{s_4}$.
  It is evident that $f_1,f_2\in \mfs_k\big(\Gamma_0(3)\big)$, ${\rm ord}_{[\infty]} \big(f_1;\Gamma_0(3)\big)=r_3+2r_4$ and ${\rm ord}_{[\infty]}
  \big(f_2;\Gamma_0(3)\big)=s_3+2s_4$. We have the following two cases:
  \begin{description}
   \item {\bf (i)} if $r_3+2r_4\neq s_3+2s_4$, then $f_1$ and $f_2$ are linearly independent in $\mfs_k(\Gamma_0(3))$.
   \item {\bf (ii)} if $r_3+2r_4 = s_3+2s_4$, then $2r_2+r_3=2s_2+s_3$. Without loos of generality, suppose that $r_3\leq s_3$ (and hence
   $s_2\leq r_2$ and $s_4\leq r_4$). Therefore, the equation $\tot\tof=\toth^{2}$ implies:
   \[
    f_1-f_2=\tot^{s_2}\toth^{r_3}\tof^{s_4}\big( \tot^{r_2-s_2}\tof^{r_4-s_4}-\toth^{s_3-r_3} \big)=0 \Longrightarrow f_1=f_2,
   \]
   or equivalently $f_1+\mathscr{I}=f_2+\mathscr{I}$.
  \end{description}
 For any integer $j$ satisfying $0\leq j \leq \left\lfloor \frac{k}{3} \right\rfloor$ we define:
 \[
  \mathscr{B}_{k,j}:=\left\{\tot^{a}\toth^{b}\tof^{c}\ | \ a,b,c\in \Z_{\geq 0}, \ 2a+4b+6c=k \ {\rm and} \ b+2c=j \right\}.
 \]
 Observe that $\mathscr{B}_{k,j}$ is non-empty and {\bf (ii)} implies that for any $g_1,g_2\in \mathscr{B}_{k,j}$, $g_1=g_2$, i.e., $\#\mathscr{B}_{k,j}=1$.
 If $0\leq j_1,j_2 \leq \left\lfloor \frac{k}{3} \right\rfloor$ and $j_1\neq j_2$, then {\bf (i)} implies that elements of $\mathscr{B}_{k,j_1}$ and
 $\mathscr{B}_{k,j_2}$ are linearly independent. Hence,
 $$\mathscr{B}_k:=\bigcup_{j=0}^{\left\lfloor \frac{k}{3} \right\rfloor}\mathscr{B}_{k,j}\subset  \mfs_k(\Gamma_0(3))$$
 is a linearly independent subset and $\#\mathscr{B}_k=\left\lfloor \frac{k}{3} \right\rfloor+1$. On account of \eqref{eq dim MkG0(3)} we deduce that
 $\mathscr{B}_k$ forms a basis for $\mfs_k(\Gamma_0(3))$, and this completes
 the proof of part 2.
  \item The proof of this part is also analogous to the proof of part 3 of Theorem \ref{thm main}.
\end{enumerate}

\section{More analogies} \label{section ma}

The $\mf$-function ($\mf$-invariant) is an important tool in the classification of elliptic curves, and is defined as the following (weight $0$) modular function for $\SL2$:
\[
{\mf}(q):=\frac{E_4^3(q)}{\Delta}=\frac{1}{q}+744+196884q+21493760q^2+864299970q^3+20245856256q^4+\ldots \,.
\]
In this way, we find similarly the following modular functions for $\Gamma_0(2)$ and $\Gamma_0(3)$, respectively:
\begin{align*}
&\mf_2=\frac{\ttt^4}{\md_2}=\frac{1}{q}+104+4372q+96256q^2+1240002q^3+10698752q^4+74428120q^5+\ldots\,,\\
&\mf_3=\frac{\tot^3}{\md_3}=\frac{1}{q}+42+783q+8672q^2+65367q^3+371520q^4+1741655q^5+\ldots\,.
\end{align*}

If we eliminate the variables $t_2$ and $t_3$ in the Ramanujan system \eqref{eq ramanujan}, then we get the Chazy equation:
\[
  2y'''-2yy''+3\big(y'\big)^2=0,
\]
which is satisfied by $E_2$.
Analogously, by  eliminating the variables $t_2$ and $t_3$ from the system \eqref{eq mvf R2 3} we obtain the differential equation:
       \begin{equation} \label{eq chazy-type}
        y'''\big(16y'-2y^2\big)-y''\big(8y''+12yy'-2y^3\big)+\big(y'\big)^2\big(20y'-3y^2\big)=0,
       \end{equation}
which is satisfied by the quasi-modular form $\tto$. We will call \eqref{eq chazy-type} a \emph{Chazy-type differential equation} for $\tto$.
The author didn't do the computations for the system \eqref{eq mvf R1 3}, but he believes that with a little more effort one can get the Cahazy-typ
equation in this case as well. Note that in the above differential equations we are substituting $t_1$ by the free parameter $y$.

In Section \ref{section pbf} we observed that the derivative of a modular form is not necessarily a modular form. However, in the case of full 
modular forms we have the Ramanujan-Serre derivation which preserves the modularity. More precisely, if 
$f\in \mfs_k\big( \SL2\big)\subset\C[E_4,E_6]$, then its
Ramanujan-Serre derivative is defined as follows which is a modular form of weight $k+2$:
\begin{equation}
\partial f:=f'-\frac{k}{12}E_2f=-\frac{1}{3}E_6\frac{\partial f}{\partial E_4}-\frac{1}{2}E_4^2\frac{\partial f}{\partial E_6}\in \mfs_{k+2}\big( \SL2\big).
\end{equation}
Similarly we get the Ramanujan-Serre-type derivations $\partial_2$ and $\partial_3$ for $\Gamma_0(2)$ and $\Gamma_0(3)$, respectively, 
which preserve the modularity. In fact, if $f\in \mfs_k\big( \Gamma_0(2) \big)\subset \C[\ttt,\ttth]$ or $f\in \mfs_k\big( \Gamma_0(3) \big)
\subset \C[\tot,\toth,\tof]/\tilde{\mathscr{J}}$, then we define 
$\partial_2f$ or $\partial_3f$, respectively, as follows:
\begin{align}
  &\partial_2 f:=f'-\frac{k}{8}\tto f=-\frac{1}{4}\ttth \frac{\partial f}{\partial \ttt}-\frac{1}{2}\ttt^3\frac{\partial f}{\partial \ttth}\in \mfs_{k+2}\big( \Gamma_0(2)\big), \\
  & \partial_3 f:=f'-\frac{k}{6}\too f=\frac{1}{3}\big( -\tot^2+54\toth \big)\frac{\partial f}{\partial \tot}+\big( \frac{1}{3}\tot\toth +9\tof \big) \frac{\partial f}{\partial \toth}\\
  &\qquad\qquad\qquad\qquad\qquad\qquad\qquad\qquad\qquad\, +\tot\tof \frac{\partial f}{\partial \tof}\in \mfs_{k+2}\big( \Gamma_0(3)\big). \nonumber
\end{align}

All quasi-modular forms introduced in this paper have integer Fourier coefficients, and some of them are listed in the On-line Encyclopedia
of Integer Sequences \cite{oeis}. In the first and the second row of the following table we give, respectively, the modular forms and the corresponding reference number in \cite{oeis}. One can find more information about these modular forms in the referred webpage and references
therein.
{
\begin{equation} \nonumber
\label{}
\begin{array}{|c|c|c|c|c|c|c|c|}
\hline
\ttt & \md_2 & \mf_2& \tot & \toth &  \tof & \md_3 &\mf_3 \\
\hline
 A004011 & A002288 & A007267 & A008653 & A198956 & A198958 & A007332 & A030197\\
\hline
\end{array}
\end{equation}
}

\section{Applications} \label{section appl}

In this section using the Fourier coefficients of $\tto,\too,\ttt,\tot,\ttth,\toth,\md_2,\md_3$ and the relations between them we find some interesting
and non-trivial formulas and congruences. For any non-negative integers $k$ and $n$ we define:
\[
 \delta_k^n:=\left \{
\begin{array}{l}
1,\,\, \quad  \textrm{\rm if} \  n\equiv 0 \ ({\rm mod} \ k);
\\
0,\,\, \quad  \textrm{\rm if} \  n\not\equiv 0 \ ({\rm mod} \ k).
\end{array} \right.
\]
In the rest of this article we denote the divisor (sigma) function by $\sigma(k):=\sigma_1(k)=\sum_{d|k}d, \ k\in \Z_{>0}$. We also define $\sigma(0):=-\frac{1}{24}$, hence we can write:
\begin{equation}
E_2(q)=-24\sum_{n=0}^{\infty} \sigma(n)q^n\,,
\end{equation}
from which we get:
\begin{align}
 \tto&=\sum_{n=0}^\infty -8\bigg(\sigma(n)+2\delta_2^n\sigma\left(\frac{n}{2}\right)\bigg)q^n\,, \label{eq P_2}\\
 \too&=\sum_{n=0}^\infty -6\bigg(\sigma(n)+3\delta_3^n\sigma\left(\frac{n}{3}\right)\bigg)q^n\,. \label{eq P_3}
\end{align}

\subsection{A recurrence relation for $\sigma$} \label{}

Using the first equation of \eqref{eq mvf R2 3} we get the following non-trivial recurrence formula for sigma function.
Similarly one can obtain other relations from other equations of the systems \eqref{eq mvf R2 3}
and \eqref{eq mvf R1 3}.

\begin{prop} \label{prop 1}
For any integer $k\geq 1$ we have:
\begin{align}
  & \sigma(2k)=\frac{8}{2k-1}\sum_{j=1}^{k-1}\sigma(j)\big( 2\sigma(2k-j)  +4\sigma(k-j)-5\sigma(2k-2j)  \big) \label{eq sigma(2k)}\\
  &\qquad\quad -\frac{4k+1}{2k-1}\sigma(k)+\frac{8}{2k-1}\big(\sigma(k)\big)^2,\nonumber \\
  & \sigma(2k+1)=\frac{1}{k}\sum_{j=1}^{k}\sigma(j)\big( 8\sigma(2k+1-j)-20\sigma(2k+1-2j) \big). \label{eq sigma(2k+1)}
\end{align}
\end{prop}

\begin{proof}
  On account of the first equation of the system \eqref{eq mvf R2 3} we have:
  \[
  \tto'=\frac{1}{8} \big( \tto^2-\ttt^2 \big).
  \]
  After comparing the Fourier coefficients in both sides of this equality we obtain the desired recurrences.
\end{proof}
Using \eqref{eq sigma(2k)} and \eqref{eq sigma(2k+1)}, for any integer $k\geq 1$, we find the following congruences:
\begin{align}
 &(2k-1)\sigma(2k)\equiv (4k+7)\sigma(k) \ ({\rm mod} \ 8)\,, \ \ \sigma(2k)\equiv \sigma(k) \ ({\rm mod} \ 2)\,,\\
 &k\sigma(2k+1)\equiv 0 \ ({\rm mod} \ 4)\, .
\end{align}

\subsection{Ramanujan tau function} \label{subsection tau}

The Ramanujan tau function, namely $\tau(n), \ n\in \N$, is defined as the Fourier coefficients of the modular discriminant $\Delta$, i.e.:
\begin{equation}
 \Delta=\frac{E_4^3-E_6^2}{1728}=\eta^{24}(q)=\sum_{n=1}^\infty \tau(n)q^n\,.
\end{equation}
A vast number of works have been dedicated to this function, see for instance \cite{beon,bcot} and references therein. Here we also give some properties
for $\tau(n)$ and then find analogous properties for Fourier coefficients of $\Delta_2$ and $\Delta_3$, namely $\tau_2$ and
$\tau_3$.

The equation $\Delta'=\Delta E_2$ yields the relation:
\begin{equation}
 \tau(n)=-\frac{24}{n-1}\sum_{j=1}^{n-1} \tau(j) \sigma(n-j)\,, \ \ n\geq 2\,,
\end{equation}
which implies:
\begin{equation}
 (n-1)\tau(n)\equiv 0 \ ({\rm mod} \ 24)\,, \ \ n\geq 2 \, ,
\end{equation}
and in particular:
\begin{equation}
 \tau(n)\equiv 0 \ ({\rm mod} \ 24)\,, \ {\rm if} \  {\rm g.c.d}(n-1,6)=1\,.
\end{equation}

Using the Ramanujan system one can check that the following differential equation holds:
\begin{equation} \label{eq E2(5)}
 4E_2^{(5)}-10E_2E_2^{(4)}+100E_2'E_2'''-100(E_2'')^2=144\Delta.
\end{equation}
Alternately, one can obtain the above equation from the fact that $[E_2,E_2]_4-4E_2^{(5)}\in \cfs_{12}\big(\SL2\big)$ (see \cite[Proposition 5.3.27]{cs17}),
and that $\cfs_{12}\big(\SL2\big)$ is $1$-dimensional and generated by $\md$. Then, by comparing the Fourier coefficients of the equation \eqref{eq E2(5)}
we obtain:
\begin{equation}
 \tau(n)=40\sum_{j=1}^{n-1}\Big( \big(-10n^2j^2+30nj^3-21j^4\big)\sigma(j)\sigma(n-j) \Big) -\frac{1}{3} n^4(2n-5)\sigma(n)\,.
\end{equation}
This relation implies:
\begin{align}
 \tau(n)&\equiv n^4\sigma(n)\equiv n\sigma(n) \ ({\rm mod} \ 2)\,, \\
 \tau(n)&\equiv n^5\sigma(n) \equiv n\sigma(n) \ ({\rm mod} \ 5)\, ,
\end{align}
and in particular implies the following known congruence relations of Ramanujan:
\begin{align*}
 \tau(2k)&\equiv 0 \ ({\rm mod} \ 2)\, \ \& \ \, \tau(3k)\equiv 0 \ ({\rm mod} \ 3)\, \ \& \ \, \tau(5k)\equiv 0 \ ({\rm mod} \ 5)\, .
\end{align*}

\subsection{Ramanujan-type tau function $\tau_2$}

We define the Ramanujan-type tau function $\tau_2$ to be the Fourier coefficients of $\md_2$, i.e.:
\begin{equation}
 \Delta_2(q)=\sum_{n=0}^\infty \tau_2(n)q^n=\eta^8(q)\eta^8(q^2)\in \cfs_8\big(\Gamma_0(2)\big)\,.
\end{equation}

Using the equations $\Delta_2'=\Delta_2\tto$ and \eqref{eq P_2}, we obtain the following recursion formula:
\begin{align}
 &\tau_2(1)=1\,, \nonumber \\
 &\tau_2(n)=-\frac{8}{n-1}\sum_{j=1}^{n-1}\tau_2(j) \bigg(\sigma(n-j)+2\delta_2^{n-j}\sigma\Big(\frac{n-j}{2}\Big)\bigg)
 \,, \ \ n\geq 2\,,
\end{align}
which implies:
\begin{equation}
 (n-1)\tau_2(n)\equiv 0 \ ({\rm mod} \ 8)\,, \ \ n\geq 2 \, .
\end{equation}
In particular, since ${\rm g.c.d}(2k-1,8)=1$, for any integer $k\geq 1$, we get:
\begin{equation}\label{eq tau2mod8=0}
 \tau_2(2k)\equiv 0 \ ({\rm mod} \ 8)\, .
\end{equation}

Using the system \eqref{eq mvf R2 3} we obtain that $\tto$ and $\md_2$ satisfy the following differential equation:
\begin{equation}
 -6\tto\tto''+9(\tto')^2+4\tto'''=16\Delta_2,
\end{equation}
which also can be deduced from $4[\frac{1}{8}\tto,\frac{1}{8}\tto]_2-2\big(\frac{1}{8}\tto\big)'''\in \cfs_{8}\big(\Gamma_0(2)\big)$ (see \cite[Theorem 2.1]{younes4}),
and that $\cfs_{8}\big(\Gamma_0(2)\big)$ is $1$-dimensional and is generated by $\md_2$. Using this equation we find $\tau_2(n), \ n\geq 2$, as follows:
\begin{align}
 \tau_2(n)&=12\sum_{j=1}^{n-1} \left\{ \Big(3nj-5j^2 \Big) \left(\sigma(j)+2\delta_2^{j}\sigma\Big(\frac{j}{2}\Big)  \right) 
 \left( \sigma(n-j)+2\delta_2^{n-j}\sigma\Big(\frac{n-j}{2}\Big) \right)
  \right\} \\
     &\ +\Big(3n^2-2n^3\Big)\bigg(\sigma(n)+2\delta_2^{n}\sigma\left(\frac{n}{2}\right)\bigg)\,. \nonumber
\end{align}
In particular we obtain:
\begin{align}
 \tau_2(n)&\equiv  n^2 \sigma(n)\equiv  n \sigma(n) \ ({\rm mod} \ 2)\,, \\
 \tau_2(n)&\equiv n^3\bigg(\sigma(n)+2\delta_2^{n}\sigma\left(\frac{n}{2}\right)\bigg)\equiv n\bigg(\sigma(n)+2\delta_2^{n}\sigma\left(\frac{n}{2}\right)\bigg) \ ({\rm mod} \ 3)\, .
\end{align}
which imply:
\begin{align}
 \tau_2(2k)&\equiv 0 \ ({\rm mod} \ 2)\,, \ \ k\geq 1\,, \\
 \tau_2(3k)&\equiv 0 \ ({\rm mod} \ 3)\,, \ \ k\geq 1\, . \label{eq tau2mod3=0}
\end{align}
Equations  \eqref{eq tau2mod8=0} and \eqref{eq tau2mod3=0} yield:
\begin{equation}\label{eq tau2mod24=0}
 \tau_2(6k)\equiv 0 \ ({\rm mod} \ 24)\,, \ \ k\geq 1 \, .
\end{equation}

\subsection{Ramanujan-type tau function $\tau_3$}
We consider the Ramanujan-type tau function $\tau_3$ as Fourier coefficients of $\md_3$, i.e.:
\begin{equation}
 \Delta_3(q)=\sum_{n=0}^\infty \tau_3(n)q^n=\eta^6(q)\eta^6(q^3)\in \cfs_6(\Gamma_0(3))\,.
\end{equation}

Using the equation $\Delta_3'=\Delta_3\too$ and \eqref{eq P_3}, we get the following recursion formula:
\begin{align}
 &\tau_3(1)=1\,,\nonumber \\
 &\tau_3(n)=-\frac{6}{n-1}\sum_{j=1}^{n-1} \tau_3(j) \left(\sigma(n-j)+3\delta_3^{n-j}\sigma\Big(\frac{n-j}{3}\Big)\right)
 \,, \ \ n\geq 2\,,
\end{align}
which implies:
\begin{equation}
 (n-1)\tau_3(n)\equiv 0 \ ({\rm mod} \ 6)\,, \ \ n\geq 2 \, .
\end{equation}
In particular, if ${\rm g.c.d}(n-1,6)=1$, then:
\begin{equation}\label{eq tau3mod6=0}
 \tau_3(n)\equiv 0 \ ({\rm mod} \ 6)\,, \ \ n\geq 2 \, .
\end{equation}

Using the system \eqref{eq mvf R1 3} we can verify that $\too$, $\tot$ and $\md_3$ satisfy the following differential equation:
\begin{equation} \label{eq 3order de P_3}
 \too'''-2\too\too''+3(\too')^2=6\tot\Delta_3,
\end{equation}
Note that similarly to $\tau_2$, this equation can also be checked from $4[\frac{1}{6}\too,\frac{1}{6}\too]_2-2\big(\frac{1}{6}\too\big)'''\in \cfs_{8}\big(\Gamma_0(3)\big)$
(see \cite[Theorem 2.1]{younes4}),
and that $\cfs_{8}\big(\Gamma_0(3)\big)$ is $1$-dimensional and is generated by $\tot\md_3$.
After computing the $q$-expansion of $\tot$ as follows:
\begin{align}
 \tot&=\sum_{n=0}^\infty 12\bigg( \sigma(n)-3\delta_3^n\sigma\left(\frac{n}{3}\right) \bigg)q^n\,.
\end{align}
and using \eqref{eq 3order de P_3} we find $\tau_3(n), \ n\geq 2$, recursively as follows:
\begin{align}
 \tau_3(n)&=\sum_{j=1}^{n-1} \left\{ 6\left(3nj-5j^2\right) \left(\sigma(j)+3\delta_3^j\sigma\Big(\frac{j}{3}\Big)  \right)
 \left(\sigma(n-j)+3\delta_3^{n-j}\sigma\Big(\frac{n-j}{3}\Big) \right)   \right.\\
  &\qquad \ \ \ \left. -12 \tau_3(j) \left(\sigma(n-j)-3\delta_3^{n-j}\sigma\Big(\frac{n-j}{3}\Big) \right) \right\} \nonumber \\
     &-n^2\big(n-2\big)\left(\sigma(n)+3\delta_3^n\sigma\left(\frac{n}{3}\right)\right)\,. \nonumber
\end{align}
Hence we obtain:
\begin{align}
 \tau_3(n)&\equiv  5n^2\left(n-2\right)\left(\sigma(n)+3\delta_3^n\sigma\left(\frac{n}{3}\right)\right)  \ ({\rm mod} \ 6) \,, \\
 \tau_3(n)&\equiv n^3\left(\sigma(n)+\delta_3^n\sigma\left(\frac{n}{3}\right)\right)\equiv n\left(\sigma(n)+\delta_3^n\sigma\left(\frac{n}{3}\right)\right) \ ({\rm mod} \ 2)\, ,\\
 \tau_3(n)&\equiv 2n^2\left(n-2\right)\sigma(n)\equiv 2n\left(n+1\right)\sigma(n) \ ({\rm mod} \ 3)\, .
\end{align}
In particular we obtain:
\begin{align}
 \tau_3(6k)&\equiv 0 \ ({\rm mod} \ 6)\,, \ \ k\geq 1\,, \\
 \tau_3(2k)&\equiv 0 \ ({\rm mod} \ 2)\,, \ \ k\geq 1\,, \\
 \tau_3(3k)&\equiv 0 \ ({\rm mod} \ 3)\,, \ \ k\geq 1\, . \label{eq tau3mod3=0}
\end{align}




{
\def\cprime{$'$} \def\cprime{$'$} \def\cprime{$'$}

}


\end{document}